\documentclass[reqno]{amsart}

\usepackage{amsmath}
\usepackage{amsthm}
\usepackage{amsfonts}
\usepackage{amssymb}
\usepackage{enumitem}
\usepackage{hyperref}

\newcommand{\indicator}[1]{\ensuremath{\mathbf{1}_{\{#1\}}}}
\newcommand{\oindicator}[1]{\ensuremath{\mathbf{1}_{{#1}}}}

\DeclareMathOperator{\cov}{Cov}
\DeclareMathOperator{\var}{Var}

\newcommand{\Prob}{\mathbb{P}}
\newcommand{\E}{\mathbb{E}}
\newcommand{\C}{\mathbb{C}}
\renewcommand{\P}{\mathbb{P}}
\renewcommand\Re{\operatorname{Re}}
\renewcommand\Im{\operatorname{Im}}
\newcommand{\eps}{\varepsilon}
\newcommand*\lap{\mathop{}\!\Delta}

\def\R{\mathbb{R}}

\theoremstyle{plain}
  \newtheorem{theorem}{Theorem}[section]
  
  \newtheorem{lemma}[theorem]{Lemma}
  \newtheorem{corollary}[theorem]{Corollary}

\theoremstyle{definition}
  \newtheorem{definition}[theorem]{Definition}
  
  \newtheorem{assumption}[theorem]{Assumption}

\theoremstyle{remark}
  \newtheorem{remark}[theorem]{Remark}

\newcommand{\set}[1]{\left\{#1\right\}} % set
\newcommand{\abs}[1]{\left\vert#1\right\vert} % absolute value

\begin{document}

\title[Partial linear eigenvalue statistics]{Partial linear eigenvalue statistics for non-Hermitian random matrices}

\author{Sean O'Rourke}
\address{Department of Mathematics\\ University of Colorado\\ Campus Box 395\\ Boulder, CO 80309-0395\\USA}
\email{sean.d.orourke@colorado.edu}
\thanks{S. O'Rourke has been supported in part by NSF grants ECCS-1610003 and DMS-1810500.}

\author{Noah Williams}
\address{Department of Mathematical Sciences\\ Appalachian State University\\ 342 Walker Hall\\ 121 Bodenheimer Dr\\ Boone, NC 28608\\USA}
\email{williamsnn@appstate.edu}

\begin{abstract}
For an $n \times n$ independent-entry random matrix $X_n$ with eigenvalues $\lambda_1, \ldots, \lambda_n$, the seminal work of Rider and Silverstein \cite{RS} asserts that the fluctuations of the  linear eigenvalue statistics $\sum_{i=1}^n f(\lambda_i)$ converge to a Gaussian distribution for sufficiently nice test functions $f$.  We study the fluctuations of $\sum_{i=1}^{n-K} f(\lambda_i)$, where $K$ randomly chosen eigenvalues have been removed from the sum.  In this case, we identify the limiting distribution and show that it need not be Gaussian.  Our results hold for the case when $K$ is fixed as well as the case when $K$ tends to infinity with $n$.  

The proof utilizes the predicted locations of the eigenvalues introduced by E. Meckes and M. Meckes \cite{MM}.  As a consequence of our methods, we obtain a rate of convergence for the empirical spectral distribution of $X_n$ to the circular law in Wasserstein distance, which may be of independent interest.  

\end{abstract}

\maketitle

\section{Introduction}

Suppose $X_n$ is an $n \times n$ matrix with entries in $\mathbb{C}$ and eigenvalues denoted $\lambda_1(X_n), \ldots, \lambda_n(X_n) \in \mathbb{C}$ (counted with algebraic multiplicity).  Let $\mu_{X_n}$ be the empirical spectral measure of $X_n$ defined by
\[ \mu_{X_n} := \frac{1}{n} \sum_{i=1}^n \delta_{\lambda_i(X_n)}, \]
where $\delta_z$ is a unit point mass at $z$.  

The well-known circular law asserts that when the entries of $X_n$ are independent and identically distributed (iid) copies of a random variable with mean zero and unit variance, the empirical spectral measure $\mu_{X_n/\sqrt{n}}$ of $X_n/\sqrt{n}$ converges almost surely to $\mu_{\mathrm{disk}}$, the uniform probability measure on the unit disk centered at the origin in the complex plane.  This was established in a series of papers \cite{B, E, Ginibre, Girko, GT, Mehta,TVcirc}, with the general case stated above being obtained by Tao and Vu \cite{TVESD}; we refer the reader to the survey \cite{BC} and references therein for more complete bibliographical details.  

After studying the limiting distribution, the next natural question concerns the fluctuations of $\int f d \mu_{X_n/\sqrt{n}}$ for an appropriate choice of test function $f$.  We define $S_n[f](X_n)$ to be the centered linear spectral statistic  
\begin{equation} \label{def:Snf}
	S_n[f](X_n) := \sum_{i=1}^n f(\lambda_i(X_n)) - \E\left[\sum_{i=1}^n f(\lambda_i(X_n))\right]
\end{equation}
associated to the $n \times n$ matrix $X_n$ and the test function $f$.  

In the case when the entries of the $n \times n$ matrix $X_n$ are iid random variables with mean zero, unit variance, and which satisfy some additional moment and regularity requirements, Rider and Silverstein \cite{RS} showed that $S_n[f](X_n/\sqrt{n})$ converges in distribution to a Gaussian random variable as $n \to \infty$ for test functions $f$ analytic in a neighborhood of the disk $\{z \in \mathbb{C} : |z| \leq 4\}$.  This result has been extended and generalized in subsequent works; see, for example, \cite{CES, CES2, CO, Jana, Kopel, KOV, OR15, RV}.  

In this paper, we focus on the fluctuations of the partial linear eigenvalue statistics 
\[ \sum_{i=1}^{n-K} f(\lambda_i(X_n/\sqrt{n} )), \]
where $K$ randomly chosen eigenvalues have been removed from the sum.  In contrast to the results cited above, we show that the limiting distribution is no longer Gaussian in this case.  
This phenomenon was first observed by Johansson (see Remark 2.1 of \cite{J}) for random unitary matrices.  
Results for partial linear eigenvalue statistics of Hermitian random matrices have previously appeared in \cite{BPZ, OS}.  
Limit laws and other results are also known for thinned point processes coming from random matrix theory, see \cite{BD, BP2, BP, G} and references therein.  
In this paper, we consider an ensemble of non-Hermitian random matrices with independent entries.  To the best of the authors' knowledge, no results are known for the partial linear eigenvalue statistics of this ensemble.  
%To the best of the authors' knowledge, no results are known for partial linear eigenvalue statistics of independent-entry matrices.  

\subsection{The model and notation}

We focus on the following model of random matrices with iid entries.  

\begin{definition}[iid random matrix] \label{def:ie}
An \emph{iid matrix} is a random $n \times n$ matrix $X_n = (x_{ij})$ (or more precisely a sequence $X_1, X_2, \ldots$ of such matrices) whose entries $x_{ij}$, $i, j \geq 1$ are independent copies of a complex-valued random variable $\xi$.  In this case, $\xi$ is called the \emph{atom variable} (or \emph{atom distribution}) of $X_n$.  
\end{definition}

There are many examples of iid matrices.  The case when the entries of $X_n$ are iid with the standard complex normal distribution is known as the complex Ginibre ensemble.  The real Ginibre ensemble is similarly defined when the entries of $X_n$ are iid with the real standard normal distribution.  The Bernoulli-Rademacher case, when the iid entries take the values $\pm 1$ with equal probability, provides another example.  

We consider iid matrices $X_n$ whose atom distribution $\xi$ satisfies the following assumptions.

\begin{assumption} \label{assump:moments}
We assume that $\xi$ has mean zero and unit variance.  In addition, we assume $\xi$ has finite moments of all orders, i.e., for any $p \in \mathbb{N}$, there is a constant $C_p > 0$, such that
\[ \E| \xi|^p \leq C_p. \]
\end{assumption}

Our main results focus on a class of test functions with polynomial growth at infinity.
\begin{definition}[Functions with polynomial tails]
	We say that the function $f:\C \to \C$ has a \textit{polynomial tail} if there exists a constant $C > 0$ and a natural number $m$ so that
	\[
	\abs{{f(z)}} \leq C(1 + |z|^m)
	\]
	\label{def:polyTails}
	for all $z \in \mathbb{C}$.  
\end{definition}

Throughout the paper, we use asymptotic notation (such as $O, o, \ll$) under the assumption that $n \to \infty$.  We use $U=O(V)$, $V=\Omega(U)$, $U \ll V$, or $V \gg U$ to denote the estimate $|U| \leq C |V|$ for some constant $C > 0$ independent of $n$ and all $n \geq C$.  If $C$ depends on a parameter, e.g., $C = C_k$, we will indicate this with subscripts, e.g., $U = O_k(V)$.  We write $U = \Theta(V)$ if $U \ll V \ll U$.  We write $U = o(V)$ if $|U| \leq a_n V$ for some sequence $a_n$ that goes to zero as $n \to \infty$.  We allow the implicit constant $C$ and the sequence $a_n$ in our asymptotic notation to depend on the constants $C_p$, $p \in \mathbb{N}$ from Assumption \ref{assump:moments} without denoting this dependence.  We say an event $E$ (which depends on $n$) holds with \emph{overwhelming probability} if for every $\alpha > 0$, $\Prob(E) \geq 1 - O_\alpha(n^{-\alpha})$.  

We denote the discrete interval $[n]:= \set{1, 2, \ldots, n}$.  $\sqrt{-1}$ denotes the imaginary unit, and we reserve $i$ as an index.  $\mu_{\mathrm{disk}}$ will denote the uniform probability measure on the unit disk centered at the origin in the complex plane.  We use $d^2 z$ to denote integration with respect to the Lebesgue measure on $\mathbb{C}$, e.g., $\int_{\mathbb{C}} f(z) d^2 z$; for complex line integrals, we integrate against $dz$, e.g., $\oint_{\mathcal{C}} f(z) dz$.  For a set $S$, $|S|$ denotes the cardinality of $S$ and $S^c$ is the complement.  For an event $E$, $\oindicator{E}$ is the indicator function of $E$.  For a square integrable random variable $\xi$, $\var(\xi)$ is its variance.  More generally, for two square integrable (real-valued) random variables $\xi, \psi$, their covariance is denoted $\cov(\xi, \psi)$ and defined as
\[ \cov(\xi, \psi) := \E [ \left(\xi - \E [\xi] \right) \left(\psi - \E [\psi] )\right ]. \]

\subsection{Partial linear eigenvalue statistics}
%This section is devoted to our results for the partial linear eigenvalue statistics of independent-entry matrices.  
For the remainder of the paper we will need to fix an ordering for the eigenvalues $\lambda_1(X_n), \ldots, \lambda_n(X_n) \in \mathbb{C}$.  Any ordering will suffice (e.g., one can first order by magnitude, and in the event of a tie, order by the argument).  We consider the fluctuations of
\[ \sum_{i \in [n] \setminus I_n} f( \lambda_i(X_n/\sqrt{n} ) ), \]
where $I_n \subset [n]$ is a random set of cardinality $|I_n| = K_n$.  In other words, we will consider the linear eigenvalue statistic $S_n[f](X_n/\sqrt{n})$ when $K_n$ eigenvalues are removed uniformly at random from the sum.  Our main results show that in this case, the limiting distribution need not be Gaussian.  For simplicity, we first illustrate our main results in the case when $X_n$ is a complex Ginibre matrix.

\begin{theorem}[Complex Ginibre case] \label{thm:ex}
Suppose $X_n$ is an $n \times n$ random matrix drawn from the complex Ginibre ensemble.  Let the function $f: \mathbb{C} \to \mathbb{R}$ have a polynomial tail and possess continuous partial derivatives in a neighborhood of the unit disk $\{z \in \mathbb{C} : |z| \leq 1\}$.  Let $K_n \geq 1$ be an integer sequence, and assume $I_n \subset [n]$ is chosen uniformly at random (independently from $X_n$) from among all  subsets of $[n]$ of size $K_n$.  
\begin{itemize}[leftmargin=7mm, itemindent=0mm, labelsep=1mm]
\item If for all sufficiently large $n$, $K_n = K$ is constant, then
\small
\begin{equation*}
\qquad\ \sum_{i \in [n] \setminus I_n} \!\!\!\!f(\lambda_i(X_n/\sqrt{n})) - \E \left[ \sum_{i \in [n] \setminus I_n} \!\!\!\!f(\lambda_i(X_n/\sqrt{n}))  \right]\!  \xrightarrow{n\to \infty} \mathcal{S} - \sum_{i =1}^{K}\left[f(U_i)- \E f(U_i)\right]
\end{equation*}
\normalsize
in distribution, where $\mathcal{S}, U_1, \ldots, U_K$ are jointly independent random variables, $U_1, \ldots, U_K$ have common distribution $\mu_{\mathrm{disk}}$, and $\mathcal{S}$ is a mean zero normal random variable with variance depending only on $f$ (the formula for the variance is given by \eqref{eq:rvvariance} below when the atom distribution $\xi$ is taken to be a standard complex normal random variable).  
\item If $K_n \to \infty$ as $n \to \infty$ and $K_n = O( n^{1/4 - \eps} )$ for some fixed $\eps \in (0, 1/4)$, then
\small
\[%
\quad\ \frac{1}{\sqrt{K_n}}\left(\sum_{i \in [n] \setminus I_n} f(\lambda_i(X_n/\sqrt{n})) - \E \left[ \sum_{i \in [n] \setminus I_n} f(\lambda_i(X_n/\sqrt{n})) \right]\right) \xrightarrow{n\to \infty} \mathcal{N}%
\]%
\normalsize
in distribution, where $\mathcal{N}$ has the normal distribution with mean zero and variance $\var f(U)$
	and $U$ has the uniform distribution $\mu_{\mathrm{disk}}$ on the unit disk.  
\end{itemize}
\end{theorem}

Theorem \ref{thm:ex} follows immediately by combining our main results (Theorems \ref{thm:finiteStats} and \ref{thm:infiniteStats}) below with Theorem 1.1 from \cite{RV}.  A few remarks concerning Theorem \ref{thm:ex} are in order.  Firstly, in both cases, the limiting distribution for the partial linear eigenvalue statistics differs from the limiting distribution of the full linear eigenvalue statistic $S_n[f](X_n/\sqrt{n})$ (the limiting distribution of $S_n[f](X_n/\sqrt{n})$ in this case is given by the normal random variable $\mathcal{S}$, see \cite{RV}).  Even if only one randomly selected eigenvalue is removed from the sum, the limiting distribution is no longer normal.  Secondly, Theorem \ref{thm:ex} shows that the more eigenvalues that are removed from the sum, the larger the variance will be.  In the extreme case where $K_n$ tends to infinity with $n$, this can be seen by the fact that one must normalize by a factor of $\sqrt{K_n}$ in order to obtain a limiting distribution.  (A similar phenomenon has been observed for thinned determinantal processes \cite{G}.)  Thirdly, we believe the condition $K_n = O(n^{1/4 - \eps})$ is an artifact of our proof (we require this bound for some technical estimates that appear in the proof). We anticipate that more advanced techniques or a different method may be able to relax this assumption.  

In our main results below, we extend Theorem \ref{thm:ex} to the case when $X_n$ is an iid matrix with atom distributions other than the complex normal distribution.  % that matches moments with the real or complex Ginibre ensemble.   

%
%\begin{definition}[Moment matching]
%For $k \geq 0$, we say two $n \times n$ independent-entry matrices $X_n = (x_{ij})$ and $X'_n = (x_{ij}')$ have \emph{matching moments to order $k$} if
%\[ \E \left[\Re(x_{ij})^a \Im(x_{ij})^b \right]  = \E \left[\Re(x_{ij}')^a \Im(x_{ij}')^b \right] \]
%whenever $1 \leq i,j \leq n$, $a, b \geq 0$, and $a + b \leq k$.  
%\end{definition} 

\begin{theorem}[Partial linear statistics when a fixed number of eigenvalues have been removed]
	%Suppose the $n\times n$ independent entry matrix $X_n = (x_{ij})$ and the test function $f$ together satisfy Assumption \ref{assum:CO} or Assumption \ref{assum:KOV}. 
	Suppose $X_n$ is an $n\times n$ iid random matrix whose atom distribution satisfies Assumption \ref{assump:moments}, $f:\C \to \C$ is a function with a polynomial tail that is Lipschitz continuous in a neighborhood of the disk $\set{z \in \C: \abs{z} \leq 1}$, and % and satisfying $\E\abs{x_{ij}}^{6+\tau}$ for some $\tau > 0$.
	$K \geq 1$ is a fixed integer. Suppose also that the statistic  $S_n[f](X_n/\sqrt{n})$ converges in distribution as $n \to \infty$ to the random variable $\mathcal{S}$. If $I_n \subset [n]$ is chosen uniformly at random (independently from $X_n$) from among all subsets of $[n]$ of size $K$, then
	%bounded test function that is analytic in a neighborhood of $\set{z:\abs{z} \leq 1}$, we have%  and bounded elsewhere, we have
	\begin{equation}
	\!\!\sum_{i \in [n] \setminus I_n}\!\!\!\!\! f(\lambda_i(X_n/\sqrt{n})) - \E \left[ \sum_{i \in [n] \setminus I_n}\!\!\!\!\! f(\lambda_i(X_n/\sqrt{n}))  \right]\!\! \xrightarrow{n\to \infty} \mathcal{S} - \sum_{i =1}^{K}\left[f(U_i)- \E f(U_i)\right]
	\label{eqn:FinStats1}
	\end{equation}
	\normalsize
	and
	\begin{equation}
	\sum_{i \in I_n} f(\lambda_i(X_n/\sqrt{n})) - \E \left[ \sum_{i \in I_n} f(\lambda_i(X_n/\sqrt{n})) \right]\! \xrightarrow{n\to \infty} \sum_{i =1}^{K}\left[f(U_i)- \E f(U_i)\right]
	\label{eqn:FinStats2}
	\end{equation}
	in distribution, where $\mathcal{S}$, $U_1, \ldots, U_{K}$ are jointly independent random variables and $U_1, \ldots, U_K$ have common distribution $\mu_{\mathrm{disk}}$.  %random variables with the following distributions. 
	%	\begin{itemize}
	%		\item In the case where $X_n$ and $f:\C \to \C$ satisfy Assumption \ref{assum:CO}, $\mathcal{N}$ has a complex Gaussian distribution with mean zero and covariances
	%		\textcolor{red}{\[
	%			\rm covariances\ here
	%			\]}.
	%		\item In the case where $X_n$ and $f: \C \to \R$ satisfy Assumption \ref{assum:KOV}, $\mathcal{N}$ has a Normal distribution with mean zero and variance
	%		\textcolor{red}{\[\rm variance\
	%			 here\]}.
	%	\end{itemize}
	%		In both cases, $U_1, \ldots, U_{K}$ are iid complex random variables drawn from the uniform distribution on the unit disk.
	\label{thm:finiteStats}
\end{theorem}%\todo{define $[n]$ in an appropriate place.}

\begin{theorem}[Partial linear statistics when a growing number of eigenvalues are removed]
	Suppose $X_n$ is an $n\times n$ iid random matrix whose atom distribution satisfies Assumption \ref{assump:moments}, $f:\C \to \C$ is a function with a polynomial tail that is Lipschitz continuous in a neighborhood of the disk $\set{z \in \C: \abs{z} \leq 1}$, and $K_n \geq 1$ is an integer sequence with the property that $K_n \to \infty$ as $n \to \infty$ and $K_n = O(n^{1/4-\eps})$ for some fixed $\eps \in (0,1/4)$. Suppose also that the statistic $\frac{1}{\sqrt{K_n}}S_n[f](X_n/\sqrt{n})$ converges to zero in probability. If $I_n \subset [n]$ is chosen uniformly at random (independently from $X_n$) from among all subsets of $[n]$ of size $K_n$, then
	%Suppose $X_n = (x_{ij})$ is an $n\times n$ independent-entry random matrix matching moments with the real or complex Ginibre ensemble to third order. 
	%Let $K_n \geq 1$ be such that $K_n = O(n^{1/4-\eps})$ for some $\eps \in (0,1/4)$, and choose $I_n \subset [n]$ uniformly at random from among the size-$K_n$ subsets of $[n]$ (and independently from $\set{x_{ij}}$). Then, if $f:\C \to \C$ is any Lipschitz test function that is analytic in a neighborhood of $\set{z:\abs{z} \leq 1}$  and bounded elsewhere, we have
	\begin{equation}
	\frac{1}{\sqrt{K_n}}\sum_{i \in [n] \setminus I_n} f(\lambda_i(X_n/\sqrt{n})) - \frac{1}{\sqrt{K_n}} \E \left[ \sum_{i \in [n] \setminus I_n} f(\lambda_i(X_n/\sqrt{n})) \right] \xrightarrow{n\to \infty} \mathcal{N}
	\label{eqn:InfStats1}
	\end{equation}
	and
	\begin{equation}
	\frac{1}{\sqrt{K_n}}\sum_{i \in I_n} f(\lambda_i(X_n/\sqrt{n})) - \frac{1}{\sqrt{K_n}} \E \left[ \sum_{i \in I_n} f(\lambda_i(X_n/\sqrt{n})) \right] \xrightarrow{n\to \infty} \mathcal{N},
	\label{eqn:InfStats2}
	\end{equation}
	in distribution, where $\mathcal{N}$ is the complex normal random variable with mean zero and covariances 
	\begin{align} \label{eq:Ncov}
		&\E[ \Re^2 (\mathcal{N})] = \var  (\Re f(U) ), \qquad \E [ \Im^2(\mathcal{N})] = \var ( \Im f(U) ), \\
		&\E[ \Re( \mathcal{N}) \Im(\mathcal{N}) ] = \cov( \Re f(U) , \Im f(U)  ) \label{eq:Ncov2}
	\end{align}
	and $U$ has uniform distribution $\mu_{\mathrm{disk}}$ on the unit disk in the complex plane.  
	\label{thm:infiniteStats}
\end{theorem}

\begin{remark}
The assumption that $S_n[f](X_n/\sqrt{n})$ converges in distribution in Theorem \ref{thm:finiteStats} (resp., $\frac{1}{\sqrt{K_n}} S_n[f](X_n/\sqrt{n})$ converges in probability to zero in Theorem \ref{thm:infiniteStats}) is only required to establish \eqref{eqn:FinStats1} (resp., \eqref{eqn:InfStats1}); for the conclusion in \eqref{eqn:FinStats2} (resp., \eqref{eqn:InfStats2}), this assumption is not required.  
\end{remark}

%The proof of Theorems \ref{thm:finiteStats} and \ref{thm:infiniteStats} can be found in Section \ref{sec:EStatsPf}. 

%The following corollaries highlight situations where $S_n[f](X_n/\sqrt{n})$ and $f:\C \to \C$ satisfy the hypotheses of Theorems \ref{thm:finiteStats} and \ref{thm:infiniteStats}. 

We conclude this section by specializing Theorems \ref{thm:finiteStats} and \ref{thm:infiniteStats} to a few specific examples. 
%The first of these is an immediate consequence of Theorem 2.2 from \cite{CO} and our main results above. 
%\begin{corollary}
%	Let $X_n = (x_{ij})$ be an $n\times n$ independent-entry random matrix matching moments with the real Ginibre ensemble to third order, and assume the entries $\{x_{ij} \}$ are identically distributed.  Suppose for some $\delta > 0$, the function $f:\C\to \C$ is analytic in a neighborhood of the disk $\set{z \in \C: \abs{z} \leq 1 + \delta}$, and bounded otherwise. Then the conclusions of Theorems \ref{thm:finiteStats} and \ref{thm:infiniteStats} hold with the random variable $\mathcal{S}$ being the mean zero complex Gaussian variable whose covariance structure is given by
%	\begin{align*}
%	\E[\mathcal{S}^2 ] &= -\frac{1}{4\pi^2}\oint_{\mathcal{C}}\oint_{\mathcal{C}}f(z)f(w)(zw-1)^{-2}\,dz\,dw\\
%	\E[|\mathcal{S} |^2 ] &= \frac{1}{4\pi^2}\oint_{\mathcal{C}}\oint_{\mathcal{C}}f(z)\overline{f(w)}(z\bar{w}-1)^{-2}\,dz\,d\bar{w},
%	\end{align*}
%where $\mathcal{C}$ is the contour around the boundary of the disk $\set{z \in \C: \abs{z} \leq 1 + \delta}$.
%	\label{cor:CO}
%\end{corollary}
The first corollary, which handles the case of complex-valued atom distributions, follows immediately from Theorem 2.2 in \cite{CES}, our main results, and the Sobolev embedding theorem\footnote{By identifying $\mathbb{C}$ with $\mathbb{R}^2$, the Sobolev embedding theorem can be used to show that $f$ is Lipschitz continuous in a neighborhood of the disk $\{z \in \mathbb{C} : |z| \leq 1\}$ whenever $f \in H^{s}(\mathbb{C})$ for $s > 2$.} (see, for instance, Theorem 9.17 in \cite{F}).   
\begin{corollary}
	Let $X_n$ be an $n\times n$ iid matrix whose atom variable $\xi$ satisfies Assumption \ref{assump:moments} and $\E[ \xi^2 ] = 0$.   Suppose $f:\C \to \R$ is in the Sobolev space $H^{2 + \delta}(\mathbb{C})$ for some $\delta > 0$ and has compact support. % a polynomial tail
	Then the conclusions of Theorems \ref{thm:finiteStats} and \ref{thm:infiniteStats} hold with the random variable $\mathcal{S}$ being the mean zero real Gaussian variable whose variance is given by
	\begin{equation} \label{eq:rvvariance}
	\frac{1}{4\pi}\int_{\abs{z} < 1}\abs{\nabla f(z)}^2\,d^2z + \frac{1}{2} \sum_{k \in \mathbb{Z}} |k| |\hat{f}(k)|^2 + (\E |\xi|^4 - 2) \left( \frac{1}{\pi} \int_{|z| < 1} f(z) d^2 z - \hat{f}(0) \right)^2,
	\end{equation} 
	where $\hat{f}(k)$ denotes the $k$th Fourier coefficient of the restriction of $f$ to the circle $\abs{z} = 1$:
	\begin{equation} \label{eq:Fk}
	\hat{f}(k) := \frac{1}{2\pi}\int_0^{2\pi}f \left(e^{\sqrt{-1}\theta}\right)e^{-\sqrt{-1}k\theta}\,d\theta, \qquad k \in \mathbb{Z}. 
	\end{equation} 
	%\label{assum:KOV}
	\label{cor:CES}
\end{corollary}

We note that Corollary \ref{cor:CES} does not apply when the atom variable $\xi$ is real-valued because of the obviously contradictory restrictions this would place on $\E[\xi^2]$. The next corollary, which follows from Theorem 2.2 in \cite{CES2}, deals with this case where the entries of $X_n$ are real-valued random variables.  

\begin{corollary}
Let $X_n$ be an $n\times n$ iid matrix whose real-valued atom variable $\xi$ satisfies Assumption \ref{assump:moments}.   Suppose $f:\C \to \R$ is in the Sobolev space $H^{2 + \delta}(\mathbb{C})$ for some $\delta > 0$ and has compact support. % a polynomial tail 
	Then the conclusions of Theorems \ref{thm:finiteStats} and \ref{thm:infiniteStats} hold with the random variable $\mathcal{S}$ being the mean zero real Gaussian variable whose variance is given by
	\begin{align*} 
	\frac{1}{2\pi}\int_{\abs{z} < 1} &\abs{\nabla (P_{\mathrm{sym}}f)(z)}^2\,d^2z + \sum_{k \in \mathbb{Z}} |k| |\widehat{P_{\mathrm{sym}}f}(k)|^2 \\ 
	&\qquad\qquad+ (\E |\xi|^4 - 3) \left( \frac{1}{\pi} \int_{|z| < 1} f(z) d^2 z - \hat{f}(0) \right)^2,
	\end{align*} 
	where 
	\[ (P_{\mathrm{sym}}f)(z) := \frac{f(z) + f(\bar{z})}{2} \]
	maps functions on the complex plane to their symmetrizations with respect to the real axis, and
	$\hat{f}(k)$, $\widehat{P_{\mathrm{sym}}f}(k)$ denote the $k$th Fourier coefficients of the restrictions of the functions $f$, $P_{\mathrm{sym}}f$ to the circle $\abs{z} = 1$ as defined in \eqref{eq:Fk}.  
\end{corollary}

%The next corollary follows from Theorem 3 in \cite{KOV} and our main results. 
%\begin{corollary}
%	Let $X_n = (x_{ij})$ be an $n\times n$ independent-entry matrix matching moments with the complex Ginibre ensemble to fourth order, and assume the entries $\{x_{ij} \}$ are identically distributed, satisfy the bound
%	\[ \Prob(|x_{ij}| \geq t) \leq C \exp( - c t^2 ) \qquad \text{ for all } t \geq 0 \]
%	for some constants $C, c > 0$ (independent of $n$), and have independent real and imaginary parts.  Suppose $f:\C \to \R$ is a function with at least two continuous derivatives, supported in the spectral bulk $\set{z \in \C: \tau < \abs{z} < 1 - \tau}$ for some fixed $\tau > 0$. Then the conclusions of Theorems \ref{thm:finiteStats} and \ref{thm:infiniteStats} hold with the random variable $\mathcal{S}$ being the mean zero real Gaussian variable whose variance is given by
%	\begin{equation} \label{eq:rvvariance}
%	\frac{1}{4\pi}\int_{\abs{z} < 1}\abs{\nabla f(z)}^2\,d^2z + \frac{1}{2} \sum_{k \in \mathbb{Z}} |k| |\hat{f}(k)|^2,
%	\end{equation} 
%	where $\hat{f}(k)$ denotes the $k$th Fourier coefficient of the restriction of $f$ to the circle $\abs{z} = 1$:
%	\[
%	\hat{f}(k) = \frac{1}{2\pi}\int_0^{2\pi}f \left(e^{\sqrt{-1}\theta}\right)e^{-\sqrt{-1}k\theta}\,d\theta.
%	\]
%	\label{assum:KOV}
%\end{corollary}

\subsection{Rate of convergence to the circular law} 

In the course of proving our main results, we obtain a rate of convergence for the empirical spectral measure of an iid matrix to the uniform probability measure $\mu_{\mathrm{disk}}$ on the unit disk with respect to the  Wasserstein metric.  Recall that for two probability measures $\mu$ and $\nu$ on $\mathbb{C}$, the $L_1$-Wasserstein distance between $\mu$ and $\nu$ is given by
\[ W_1(\mu, \nu) := \inf_{\pi} \int |x - y| d \pi(x,y), \]
where the infimum is over all probability measures $\pi$ on $\mathbb{C} \times \mathbb{C}$ with marginals $\mu$ and $\nu$.  
%Alternatively, when $\mu$ and $\nu$ have compact support, one has the Kantorovich--Rubinstein formulation of the Wasserstein distance that involves Lipschitz test functions (see e.g. \cite{V}, Remark 6.5 on page 95). In particular, if $\mu$ and $\nu$ have compact support, 
%\begin{equation}
%W_1(\mu,\nu) = \sup_{\varphi:\|\varphi\|_\text{Lip}}\set{\int_\C \varphi\,d\mu - \int_\C \varphi\,d\nu},
%\label{eqn:kantorovich}
%\end{equation}
%where $\|\varphi\|_\text{Lip}$ denotes the minimal Lipschitz constant for $\varphi$.

\begin{theorem}[Wasserstein distance bound] \label{thm:iem:main}
Let $X_n$ be an $n \times n$ iid random matrix whose atom distribution satisfies Assumption \ref{assump:moments}.  Then almost surely, for $n$ sufficiently large, 
\[ W_1 ( \mu_{X_n/\sqrt{n}}, \mu_{\mathrm{disk}}) \leq n^{o(1) - 1/4}, \]
where $\mu_{\mathrm{disk}}$ is the uniform probability measure on the unit disk centered at the origin.
\end{theorem}

The bound of $n^{o(1) -1/4}$ for the Wasserstein distance appears far from optimal; we include this result since it follows as a simple corollary of our methods.  For the case when $X_n$ is drawn from the complex Ginibre ensemble, it is shown in \cite{MM} that almost surely 
\begin{equation} \label{eq:W1MM}
	W_1 ( \mu_{X_n/\sqrt{n}}, \mu_{\mathrm{disk}}) = O \left( \frac{ \sqrt{\log n} }{ n^{1/4} } \right) 
\end{equation} 
using a coupling argument.  
The best bound known to date for the complex Ginibre ensemble is 
\begin{equation} \label{eq:W1CHM}
	W_1 ( \mu_{X_n/\sqrt{n}}, \mu_{\mathrm{disk}}) = O \left( \sqrt{ \frac{ {\log n} }{ n  } } \right), 
\end{equation}
which is due to Chafa\"{\i}, Hardy, and Ma\"{\i}da \cite{CHM}.

%
%\begin{example}
%Some examples of independent entry random matrices for which Theorem \ref{thm:iem:main} applies include the real Ginibre ensemble and Bernoulli random matrices, where each entry $x_{ij}$ takes the values $\pm 1$ with probability $1/2$.  
%\end{example}
%
%\begin{remark}
%The proof reveals that the same bound can be obtained for the expectation $\E W_1 ( \mu_{X/\sqrt{n}}, \mu_{\mathrm{disk}}) $.   
%\end{remark}
%
%\begin{remark}
%The assumption that the entries of $X$ have sub-exponential decay can probably be removed by first starting with a truncation argument, although this would require re-deriving many of the results in \cite{TV}.  It may be possible to remove the three matching moments assumptions by using the local law in \cite{Y}.  However, the results in \cite{Y} need to be modified to allow the function $f$ to depend on $n$.  This seems possible, but it would probably require a lot of work.  
%\end{remark}
%
%\todo[inline]{The proof can most likely be generalized to the $L_p$-Wasserstein metric}
%
%
%\subsection{Partial Linear Eigenvalue Statistics}\todo{clean this section up}
%
%Define $S_n[f](M)$ to be the centered linear statistic
%\[
%S_n[f](M) := \sum_{i=1}^nf(\lambda_i(M)) - \E\left[\sum_{i=1}^nf(\lambda_i(M))\right],
%\]
%associated with the random $n \times n$ matrix $M$ and the test function $f$, whose precise definition will depend on the ensemble from which $M$ is drawn. 
%
%We have the following results for the partial linear eigenvalue statistics formed when $K_n$ summands have been removed from the sums comprising $S_n[f](M)$.

\subsection{Overview and outline} 
The proof of Theorem \ref{thm:iem:main} is presented in Section \ref{sec:iem:main}.  The proofs of Theorems \ref{thm:finiteStats} and \ref{thm:infiniteStats} are presented in Section \ref{sec:EStatsPf} and rely on the bounds from Section \ref{sec:iem:main}.  In the appendix, we state a version of the local circular law established recently by Alt, Erd\H{o}s, and Kr\"{u}ger in \cite{AEK2} required for the proofs of our main results.

\subsection*{Acknowledgements}
The authors would like to thank Alexander Soshnikov for originally pointing out this problem.  
The authors also wish to thank L\'{a}szl\'{o} Erd\H{o}s for references and comments which improved an earlier draft of this manuscript, and we are grateful to Gaultier Lambert for providing useful references.    
The first author thanks Elizabeth Meckes and Mark Meckes for useful discussions and references.

\section{Proof of Theorem \ref{thm:iem:main}} \label{sec:iem:main}

Let $X_n$ be as in the statement of Theorem \ref{thm:iem:main}, and let $G_n$ be an $n \times n$ iid matrix drawn from the complex Ginibre ensemble.  In view of the results from \cite{MM} (see \eqref{eq:W1MM}) or \cite{CHM} (see \eqref{eq:W1CHM}) and the Borel--Cantelli lemma, it suffices to prove that
\[ W_1 ( \mu_{X_n/\sqrt{n}}, \mu_{G_n/\sqrt{n}} ) \leq n^{o(1) - 1/4} \]
with overwhelming probability.  

To bound $W_1(\mu_{X_n/\sqrt{n}}, \mu_{G_n/\sqrt{n}})$, we observe that for any permutation $\sigma$ on $\{1, \ldots, n\}$
\begin{equation} \label{eq:iem:tobound} 
	W_1(\mu_{X_n/\sqrt{n}}, \mu_{G_n/\sqrt{n}}) \leq \frac{1}{n} \sum_{k=1}^n \frac{ |\lambda_k(X_n) - {\lambda}_{\sigma(k)}(G_n)|}{ \sqrt{n} }. 
\end{equation}
The real work here is to construct an advantageous $\sigma$, which defines how we pair each eigenvalue of $X_n$ to an eigenvalue of $G_n$.  Many will pair in a nice enough way to give a good bound; for those that do not pair nicely, we will use the following ``worst-case scenario'' bound.

%\begin{theorem}[Crude bound] \label{thm:iem:crude}
%There exists $C > 1$ such that, with overwhelming probability, all the eigenvalues of $X$ and $\tilde{X}$ are contained in the disk $\{z \in \mathbb{C} : |z| < C \sqrt{n} \}$.  
%\end{theorem}
%\begin{proof}
%After a standard truncation and centering argument, the claim follows from the spectral norm bound given in \cite[Theorem 5.9]{BS}; we omit the details.  \todo{Make this a Corollary of \ref{thm:iem:specRad} or replace all together with the better result.}
%\end{proof}
%
%We also use the following, stronger, result. 
\begin{theorem}
	For any fixed $\eps > 0$, with overwhelming probability, all the eigenvalues of $X_n$ and $G_n$ are contained in the disk $\{z \in \mathbb{C} : |z| < (1+\eps) \sqrt{n} \}$.
	\label{thm:iem:specRad}
\end{theorem}
Theorem \ref{thm:iem:specRad} follows from Remark 2.2 in \cite{AEK2}; in fact, the results in \cite{AEK2} provide much greater precision for a larger class of independent-entry matrices than what is stated here.  Alternatively, Theorem \ref{thm:iem:specRad} also follows from Theorem 2.5(ii) in \cite{AEK} or Theorem 1.4 in \cite{T}.  
%\begin{proof}
%	By definition, the entries of $X_n$ and $G_n$ have have finite absolute moments of all orders. It follows from Theorem 1.4 of \cite{T} that for fixed $m \geq 1$ and $\eta > 0$, with overwhelming probability, we have
%	\[
%	\max\set{\left\|\left(X_n/\sqrt{n}\right)^m\right\|_{op},\left\|\left(G_n/\sqrt{n}\right)^m\right\|_{op}} \leq m + 1 + \eta,
%	\]
%	where $\|\cdot\|_{op}$ denotes the operator norm. By the consistency of the matrix norm $\|\cdot\|_{op}$, the spectral radii 
%	\[
%	\rho(X_n/\sqrt{n}):= \max_{1\leq i\leq n}\abs{\lambda_i(X_n/\sqrt{n})},\quad \rho(G_n/\sqrt{n}) := \max_{1\leq i\leq n}\abs{\lambda_i( G_n/\sqrt{n})}
%	\]
%	satisfy\[\rho(X_n/\sqrt{n}) \leq \left\|\left(X_n/\sqrt{n}\right)^m\right\|_{op}^{1/m},\quad \rho(G_n/\sqrt{n}) \leq \left\|\left(G_n/\sqrt{n}\right)^m\right\|_{op}^{1/m}.\]
%	Consequently, with overwhelming probability,
%	\[
%	\max\set{\rho(X_n/\sqrt{n}), \rho(G_n/\sqrt{n})} \leq (m+1 + \eta)^{1/m}.
%	\]
%	The result follows by setting $\eta = 1$ and choosing a suitably large $m$ so that $(m+2)^{1/m} < 1 + \eps$. 
%\end{proof}

A \emph{square} is a set of the form $\{ z \in \mathbb{C} : a \leq \Re(z) < b, c \leq \Im(z) < d \}$ for some real values $a,b,c,d$ which satisfy $b - a = d - c$.  In this case, the side length of the square is $b-a$ and the area is $(b-a)^2$.  The center of the square is $(a + b)/2 + \sqrt{-1} (c + d)/2$.  

For a Borel set $B \subset \mathbb{C}$, let $N(B)$ denote the number of eigenvalues of $X_n$ in $B$ and $\widehat{N}(B)$ denote the number of eigenvalues of $G_n$ in $B$.  

Fix a constant $C > 1$ for which all of the eigenvalues of $X_n$ and $G_n$ are contained in the disk $\{z \in \C: \abs{z} < C\sqrt{n}\}$ with overwhelming probability. (Such a constant exists by Theorem \ref{thm:iem:specRad}.) %to be the value from Theorem \ref{thm:iem:crude}, and 
Then, let $R$ be the circumscribed square to the disk $\{z \in \mathbb{C} : |z| < C \sqrt{n} \}$, so that the disk lies entirely inside $R$.  We partition, $R$ into disjoint sub-squares
\[ R = \bigcup_{\ell = 1}^L R_{\ell} \]
such that all sub-squares $R_{\ell}$ have the same side length $\Theta(n^{1/4})$ (and hence same area $\Theta(n^{1/2})$).  A simple area argument reveals that the sub-squares can be constructed so that $L = O(n^{1/2})$.  Our main tool is the following.  

\begin{theorem} \label{thm:iem:TV}
With the construction above, 
\begin{equation} \label{eq:iem:TV}
	\max_{\ell} |N(R_\ell) - \widehat{N}(R_{\ell}) | \leq n^{o(1) + 1/4} 
\end{equation}
with overwhelming probability.  
\end{theorem}

The proof of Theorem \ref{thm:iem:TV} is based on a local circular law result established recently by Alt, Erd\H{o}s, and Kr\"{u}ger in \cite{AEK2}.  We present the proof in Appendix \ref{sec:TV}.  

We now turn our attention to completing the proof Theorem \ref{thm:iem:main}.  We begin by constructing the permutation $\sigma$ from \eqref{eq:iem:tobound}.  We say an eigenvalue $\lambda_k(X_n)$ of $X_n$ and an eigenvalues $\lambda_j(G_n)$ of $G_n$ \emph{pair up} if $\sigma(k) = j$.  Thus, the pairing of eigenvalues will construct $\sigma$.  

First pair $\min \{ N(R_{1}), \widehat{N}(R_1) \}$ eigenvalues of $X_n$ and $G_n$ that fall in $R_1$.  The choice of the pairing between these eigenvalues in $R_1$ is arbitrary.  Then pair $\min \{ N(R_{2}), \widehat{N}(R_2) \}$ eigenvalues of $X_n$ and $G_n$ that fall in $R_2$.  Continue in this way until $\min \{ N(R_{L}), \widehat{N}(R_L) \}$ eigenvalues are paired from $R_L$.  Then pair the remaining unpaired eigenvalues of $X_n$ and $G_n$ in an arbitrary fashion.  This completely determines the permutation $\sigma$.  Given $1 \leq k \leq n$, we say $k$ is \emph{good} if $\lambda_k(X_n)$ and $\lambda_{\sigma(k)}(G_n)$ both fall in the same sub-square $R_\ell$, $1 \leq \ell \leq L$, otherwise we say $k$ is \emph{bad}\footnote{Technically, it is more precise to say that the pair $(\sigma,k)$ is good or bad since the conditions depend on the permutation $\sigma$ as well as the index $k$.  Here, we have shortened the terminology to just involve the index $k$.}.  

From this point forward, we work on the event where all the eigenvalues of $X_n$ and $G_n$ are contained in $\{z \in \mathbb{C} : |z| \leq C \sqrt{n} \}$ and \eqref{eq:iem:TV} holds.  From Theorems %\ref{thm:iem:crude}
\ref{thm:iem:specRad} and \ref{thm:iem:TV}, this event holds with overwhelming probability.  Continuing from \eqref{eq:iem:tobound}, we have
\begin{align}
	W_1(\mu_{X_n/\sqrt{n}}, \mu_{G_n/\sqrt{n}}) &\leq \frac{1}{n} \sum_{\ell = 1}^L \sum_{\substack{k : \lambda_k(X_n) \in R_{\ell} \\ k \text{ good}}} \frac{ |\lambda_k(X_n) - {\lambda}_{\sigma(k)}(G_n)|}{ \sqrt{n} } \label{eq:iem:bnd} \\
	&\qquad\qquad + \frac{1}{n} \sum_{\ell = 1}^L \sum_{\substack{k : \lambda_k(X_n) \in R_{\ell} \\ k \text{ bad}}} \frac{ |\lambda_k(X_n) - {\lambda}_{\sigma(k)}(G_n)|}{ \sqrt{n} }. \notag 
\end{align}
We bound the two terms separately.  

If $\lambda_k(X_n) \in R_\ell$ and $k$ is good, then $\lambda_{\sigma(k)}(G_n) \in R_{\ell}$.  Thus, the distance between the two eigenvalues is at most the diameter of $R_\ell$, which is $O(n^{1/4})$.  Thus, we obtain
\begin{equation} \label{eq:iem:conc1}
	\frac{1}{n} \sum_{\ell = 1}^L \sum_{\substack{k : \lambda_k(X_n) \in R_{\ell} \\ k \text{ good}}} \frac{ |\lambda_k(X_n) - {\lambda}_{\sigma(k)}(G_n)|}{ \sqrt{n} } \ll n^{-1/4}. 
\end{equation}
If $k$ is bad, then we use the fact that the eigenvalues of $X_n$ and $G_n$ are contained in the disk $\{z \in \mathbb{C} : |z| \leq C \sqrt{n} \}$ to obtain
\[  |\lambda_k(X_n) - {\lambda}_{\sigma(k)}(G_n)| \leq 2 C \sqrt{n}. \]
Therefore, using \eqref{eq:iem:TV} to bound the number of bad indices, we conclude that
\begin{equation} \label{eq:iem:conc2}
	\frac{1}{n} \sum_{\ell = 1}^L \sum_{\substack{k : \lambda_k(X_n) \in R_{\ell} \\ k \text{ bad}}} \frac{ |\lambda_k(X_n) - {\lambda}_{\sigma(k)}(G_n)|}{ \sqrt{n} } \ll \frac{L}{n} n^{o(1) + 1/4}. 
\end{equation}
As $L = O(n^{1/2})$, combining \eqref{eq:iem:conc1} and \eqref{eq:iem:conc2} with \eqref{eq:iem:bnd} completes the proof of the theorem.

\section{Proof of Theorems \ref{thm:finiteStats} and \ref{thm:infiniteStats}}\label{sec:EStatsPf} %\todo{Change the title and move somewhere else.}

This section is devoted to the proofs of Theorems \ref{thm:finiteStats} and \ref{thm:infiniteStats}.

\subsection{Tools}

In some of our calculations below, we will need to know about the random set $I_n$.  The next lemma will help us to understand its distribution. % of this random variable.
% show that certain events have probabilities tending to $1$ as $n \to \infty$. The following lemmas will help in these calculations. 

\begin{lemma}
	Let $1 \leq K_n \leq n$, and assume $I_n$ is a random subset of $[n]$ chosen uniformly from among all $K_n$-sized subsets of $[n]$.  
	Let $J_n \subset [n]$ be fixed.  
	For $j = 0, 1, \ldots, K_n$,
	\begin{equation}
	\P\left(\abs{I_n \setminus J_n} = j\right) \leq \exp\left(\frac{K_n^2}{n}\cdot\frac{1}{\sqrt{1-\frac{K_n-1}{n}}}\right)p_n(j),
	\label{eqn:NearBin1}
	\end{equation}
	where $p_n(j)$ is the probability mass function of a binomial random variable with parameters $K_n$ and $1-\abs{J_n}/n$. %Furthermore, if $K_n + |J_n| < n$, then
%	\begin{equation}
%	\P(I_n \subset [n] \setminus J_n) \geq \exp \left( - K_n\frac{\frac{\abs{J_n}+K_n}{n}}{\sqrt{1-\frac{\abs{J_n}+K_n}{n}}} \right) .  %\exp\left(\frac{-2K_n \left(\abs{J_n}+ K_n\right)}{n}\right).
%	\label{eqn:NearBin2}
%	\end{equation}
	\label{lemma:NearBin}
\end{lemma}
\begin{proof}
	If $j > n-\abs{J_n}$ or $K_n - j >\abs{J_n}$, then it is impossible to have $\abs{I_n \setminus J_n} = j$ (for this would imply that either $\abs{I_n \setminus J_n} > n-\abs{J_n} = \abs{J_n^c}$ or $\abs{I_n \cap J_n} = K_n-j > \abs{J_n}$), so in these cases, \eqref{eqn:NearBin1} trivially holds. Otherwise, we have
	\begin{align*}
	\P(\abs{I_n \setminus J_n} = j) &= \frac{\binom{n-\abs{J_n}}{j}\binom{\abs{J_n}}{K_n-j}}{\binom{n}{K_n}}\\
	&= \binom{K_n}{j}\frac{(n-\abs{J_n})!}{(n-\abs{J_n}-j)!}\cdot \frac{\abs{J_n}!}{(\abs{J_n}-K_n+j)!} \cdot \frac{(n-K_n)!}{n!}\\
	&\leq \binom{K_n}{j}\left(\frac{n-\abs{J_n}}{n-K_n+1}\right)^j\left(\frac{\abs{J_n}}{n-K_n+1}\right)^{K_n-j}\\
	&\leq \left(\frac{n}{n-K_n+1}\right)^{K_n}p_n(j).
	\end{align*}
	Using the bound $-\log(1-x) \leq x/\sqrt{1-x}$ for $x \in [0,1)$, we find
	\begin{equation}
	\begin{aligned}
	\left(\frac{n}{n-K_n+1}\right)^{K_n} &= \exp\left[-K_n\log\left(1-\frac{K_n-1}{n}\right)\right] \leq \exp\left[K_n\cdot\frac{\frac{K_n-1}{n}}{\sqrt{1-\frac{K_n-1}{n}}}\right],
	\end{aligned}
	\label{eqn:KnDraws}
	\end{equation}
	which establishes \eqref{eqn:NearBin1}. %On the other hand, 
	%\begin{align*}
	%\P(\abs{I_n \setminus J_n} = j) & \geq {K_n \choose j}\left(\frac{n-\abs{J_n}-j + 1}{n}\right)^j\left(\frac{\abs{J_n}-K_n+j+1}{n}\right)^{K_n-j}\\
	%\end{align*}
%	If $K_n + \abs{J_n} < n$, then 
%	\begin{align*}
%	\P(I_n \subset [n] \setminus J_n) &= \frac{{n-\abs{J_n} \choose K_n}}{{  n\choose K_n}} = \frac{(n-\abs{J_n})!\cdot(n-K_n)!}{(n-\abs{J_n} - K_n)!\cdot n!}\\
%	&\geq \left(\frac{n-\abs{J_n}-K_n}{n}\right)^{K_n}\\
%	&= \exp\left(K_n\log\left[1- \frac{\abs{J_n}+K_n}{n}\right]\right)\\
%	&\geq \exp\left(K_n\frac{-\frac{\abs{J_n}+K_n}{n}}{\sqrt{1-\frac{\abs{J_n}+K_n}{n}}}\right),%\\
%	%&\gg \exp\left(-2K_n\frac{\abs{J_n}+K_n}{n}\right),
%	\end{align*}
%	where, again, we have used that $\log(1-x) \geq -x/\sqrt{1-x}$ for $x \in [0,1)$. The bound in \eqref{eqn:NearBin2} follows.
\end{proof}	

We will need the following bound for the moments of the operator norm of an iid matrix.  While this bound is far from optimal, it will suffice for our purposes.  

\begin{lemma} \label{lemma:exnorm}
If $X_n$ is an $n \times n$ iid matrix whose atom distribution satisfies Assumption \ref{assump:moments}, then for any integer $m \geq 1$
\[ \E \|X_n \|_{op}^m \ll_m n^{2 + 2m}, \]
where $\| \cdot \|_{op}$ denotes the operator norm.  
\end{lemma}
\begin{proof}
By bounding the operator norm by the Frobenius norm, we deduce from Assumption \ref{assump:moments} that
\begin{align*} 
	\Prob( \|X_n\|_{op} > t) \leq \Prob \left( \max_{1\leq i,j \leq n} |x_{ij}| \geq \frac{t}{n} \right) \leq \sum_{i,j=1}^n \Prob \left( |x_{ij}| \geq \frac{t}{n} \right) \leq C_p \frac{n^{2 + p}}{t^p } 
\end{align*} 
for any $t > 0$ and any integer $p \geq 1$.  Thus, taking $p = 2m$, we conclude that
\begin{align*}
	\E \|X_n\|_{op}^m &= \int_0^\infty \Prob( \|X_n \|_{op}^m > t ) dt \\
	&\leq 1 + \int_1^\infty \Prob( \|X_n \|_{op} > t^{1/m} ) dt \\
	&\leq 1 + C_{2m} n^{2 + 2m} \int_1^\infty \frac{1}{t^2} dt \\
	&\ll_m n^{2+2m}. 
\end{align*}
\end{proof}

\subsection{Proofs of Theorems \ref{thm:finiteStats} and \ref{thm:infiniteStats}}

We prove Theorems \ref{thm:finiteStats} and \ref{thm:infiniteStats} simultaneously. To ensure that the following arguments can be directly applied in both situations, we define $K_n := \abs{I_n}= K$ in the case where $K$ is fixed. Our plan of attack will consist of several interpolations wherein we replace the eigenvalues of $X_n/\sqrt{n}$ in the sums with points from among the predicted locations introduced in \cite{MM} for the eigenvalues of an $n\times n$ matrix drawn from the complex Ginibre ensemble. To that end, we introduce the following notation (modeled after the notation from \cite{MM}). First, find a positive integer $N_n$ so that $(N_n-1)^2 \leq n \leq N_n^2$, and then, define $m_n := n-(N_n-2)^2$. Note that $n-m_n = (N_n-2)^2$ is a perfect square, and
\begin{equation}
2\sqrt{n}- 3 \leq m_n \leq 4\sqrt{n}.
\label{MM:mbds}
\end{equation}
Next, we define the predicted (or classical) locations $\widetilde{\lambda}_i$, for $1 \leq i \leq n - m_n$, as in \cite{MM}. That is, for a given $1 \leq i \leq n-m_n$, write $\ell_i := \lceil \sqrt{i} \rceil$ and $q_i := i - (\ell_i - 1)^2$, and define
\[
\widetilde{\lambda}_i := \frac{\ell_i-1}{\sqrt{n}}e^{2\pi\sqrt{-1}q_i/(2\ell_i-1)}.
\]
(For a more detailed and intuitive understanding of this construction, we direct the reader to \cite{MM}.) Finally, for $n-m_n < i \leq n$, let $\widetilde{\lambda}_{i}$ be any arbitrary deterministic value in the annulus $\set{z \in \C : \sqrt{1- \frac{m_n}{n}} \leq \abs{z} \leq 1}$; the particular choice of values will not be relevant for the proof, and one can safely choose $\widetilde{\lambda}_i = 1$ for all $n - m_n < i \leq n$. The idea here is to facilitate a coupling between random draws from $\{\widetilde{\lambda}_i : 1 \leq i \leq n \}$ and the uniform distribution on the unit disk in a fashion inspired by the methods used in \cite{MM}.  %In the arguments that follow, the definitions of $\widetilde{\lambda}_i$ when $n-m_n < i \leq n$ are irrelevant, so choose arbitrary points inside the unit disk for their locations.
The following intermediate result establishes that with overwhelming probability, most of the eigenvalues of $X_n/\sqrt{n}$ are reasonably near the predicted locations $\{\widetilde{\lambda}_i\}$. 

\begin{lemma}
	With overwhelming probability, there is a (random) permutation $\tau_n: [n] \to [n]$ and a (random) set $J_n \subset [n]$, both measurable with respect to the $\sigma$-algebra generated by $\set{x_{ij}}$, % and independent of $\{U_i^{(n)}\}$ and \todo{is $\{Y_i^{(n)}\}$ the right thing here? Not sure, but perhaps it's better to introduce these things after this lemma.}$\{Y_i^{(n)}\}$, 
	so that
	\begin{equation} \label{eq:MMpairs}
	\max_{i \in [n] \setminus J_n}\abs{\lambda_i\left(X_n/\sqrt{n}\right) - \widetilde{\lambda}_{\tau_n(i)}} \ll n^{o(1)-1/4},
	\end{equation} 
	%uniformly in $i$, for $i \in [n]\setminus J_n$. %\todo{Clarify that the asymptotic is independent of $i$.}
	and $\abs{J_n} \ll n^{o(1)+3/4}$.
	\label{lemma:MMpairs}
\end{lemma}

\begin{proof}
	We will compare the eigenvalues of $X_n/\sqrt{n}$ to the classical locations $\widetilde{\lambda}_i$ via the intermediate comparison of each collection to the (ordered) eigenvalues of an $n\times n$ matrix $G_n := (g_{ij})$ drawn from the complex Ginibre ensemble in such a way that the iid entries $g_{ij}$ are independent from the $\sigma$-algebra generated by $\{x_{ij}\}$ and $I_n$.  Our proof relies on some results from \cite{MM} by E. Meckes and M. Meckes, so for convenience, we adopt notation similar to theirs. In particular, we consider the spiral ordering of the eigenvalues of $G_n/\sqrt{n}$ defined in Step 1 of the outline of proof presented in \cite{MM}. More explicitly, we define the order $\prec$ on $\C$ by declaring that $0 \prec z$ for all $z \in \C$, and writing $w \prec z$ for $w, z \in \C \setminus\set{0}$ if any of the following conditions hold:
	%\begin{enumerate}[(i)]
	\begin{enumerate}[label=({\roman*})]
		\item $\lfloor \sqrt{n}\abs{w}\rfloor < \lfloor\sqrt{n}\abs{z}\rfloor$;
		\item $\lfloor \sqrt{n}\abs{w}\rfloor = \lfloor\sqrt{n}\abs{z}\rfloor$ and $\arg{w} < \arg{z}$;
		\item $\lfloor \sqrt{n}\abs{w}\rfloor = \lfloor\sqrt{n}\abs{z}\rfloor$, $\arg{w} = \arg{z}$, and $\abs{w} \leq \abs{z}$.\label{MM:cond3}
	\end{enumerate}
	(Note that the last inequality comprising condition \ref{MM:cond3} is irrelevant because the eigenvalues of $G_n$ have distinct argument with probability 1.) Here, we use the convention that $\arg z \in (0, 2\pi]$.  Recall that $\lambda_1(G_n/\sqrt{n}), \ldots, \lambda_n(G_n/\sqrt{n})$ are the eigenvalues of $G_n/\sqrt{n}$ (ordered in some specific but arbitrary fashion).  We now define $\set{ \lambda_1'(G_n/\sqrt{n})}_{i=1}^n$ to be the eigenvalues of $G_n/\sqrt{n}$ ordered so that $\lambda_1'(G_n/\sqrt{n}) \prec \cdots \prec \lambda_n'(G_n/\sqrt{n})$.
	
	%Let $G_n = (g_{ij})$ be an $n\times n$ matrix drawn from the complex Ginibre ensemble whose entries $g_{ij}$ are iid and independent of $\{x_{ij}\}$, $\{U_i^{(n)}\}$, and $\{Y_i^{(n)}\}$. 
	
	We first compare $\{\lambda_i(X_n/\sqrt{n}) \}$ to $\{\lambda'_i(G_n/\sqrt{n}) \}$. With overwhelming probability, we can construct a random permutation $\sigma$ as we did in the proof of Theorem \ref{thm:iem:main} above, %\todo{revise this/include lemma about $\sigma$ watch out for $G_n$ vs. $\widetilde{X}$. Should we write $\sigma_n$ to advertise dependence on $n$?}
	and define
	\[
	J'_n:=\set{1\leq k \leq n : \text{$k$ is \textit{bad}}}.
	\]
	By construction, %for $i \in [n] \setminus J'_n$ (i.e. for $i$ that are \textit{good}),
	\[
	\max_{i \in [n] \setminus J'_n}\abs{\lambda_i\left(X_n/\sqrt{n}\right) - \lambda_{\sigma(i)}\left(G_n/\sqrt{n}\right)}\ll n^{-1/4},
	\]
	so by re-labeling the eigenvalues of $G_n/\sqrt{n}$ according to $\prec$, we have the following: with overwhelming probability, there is a permutation $\tau_n \in S_n$, measurable with respect to the $\sigma$-algebra generated by $\set{x_{ij}}$ and $\set{g_{ij}}$, so that %$i \in [n] \setminus J'_n$ implies 
	\begin{equation}
	\max_{i \in [n] \setminus J'_n}\abs{\lambda_i\left(X_n/\sqrt{n}\right) - \lambda'_{\tau_n(i)} (G_n/\sqrt{n})} \ll n^{-1/4}.
	\label{eqn:MMpart1a}
	\end{equation}
	In addition, Theorem \ref{thm:iem:TV} and the fact that $L = O(n^{1/2})$ together imply that with overwhelming probability,
	\begin{equation}
	\abs{J'_n} \ll L\cdot n^{o(1) + 1/4} \ll n^{o(1) + 3/4}.
	\label{eqn:MMpart1b}
	\end{equation}
	
	Next, we compare $\{\lambda_i' (G_n/\sqrt{n}) \}$ to $\{\widetilde{\lambda}_i\}$ using the results of \cite{MM}. Fix $\alpha > 0$, and define
	$$a_n := \sqrt{512\pi^2(\alpha+1)\log{n}} \ll n^{o(1)}.$$
	By Theorem 4.3 from \cite{MM} (with $s:=a_nn^{1/4}$),  there exists an absolute constant $C>0$, so that whenever $i \leq n-m_n$ satisfies
	\begin{equation}
	\lceil \sqrt{i} \rceil \leq \sqrt{n}-\sqrt{\log{n}}
	\label{eqn:iupper}
	\end{equation}and
	\begin{equation}
	9 \leq a_nn^{1/4} \leq \pi(\lceil \sqrt{i}\rceil - 1) + 2,
	\label{eqn:ilower}
	\end{equation} we have
	\begin{equation}
	\begin{aligned}
	\P&\left(\abs{\lambda_i'(G_n/\sqrt{n})-\widetilde{\lambda}_i} > \frac{a_n}{n^{1/4}}\right) \\
	&\qquad\qquad\leq C\exp\left(-\min\set{\frac{(a_nn^{1/4}-9)^2}{256\pi^2(\lceil \sqrt{i}\rceil-1)},\frac{a_nn^{1/4}-9}{4\pi}}\right)\\
	&\qquad\qquad\leq C\exp\left(-\min\set{\frac{a_n^2n^{1/2}-18a_nn^{1/4}}{256\pi^2\sqrt{n}},\frac{a_nn^{1/4}-9}{4\pi}}\right)\\
	&\qquad\qquad\ll\exp\left(-\frac{a_n^2n^{1/2}-18a_nn^{1/4}}{256\pi^2\sqrt{n}}\right)\\
	&\qquad\qquad\ll\exp\left(-\frac{a_n^2n^{1/2}}{512\pi^2\sqrt{n}}\right) = n^{-\alpha-1}.
	\end{aligned}\label{eqn:MMprob}
	\end{equation}
	Now, $9 \leq a_nn^{1/4}$ for large $n$, and $i \geq \left(\frac{a_nn^{1/4} + \pi - 2}{\pi}\right)^2$ implies the rightmost inequality in \eqref{eqn:ilower}. In addition  inequality \eqref{MM:mbds} establishes that for large $n$, 
	\[
	\sqrt{n} - \sqrt{\log{n}} = \sqrt{n - 2\sqrt{n\log{n}} + \log{n}} \leq\sqrt{n-4\sqrt{n}} \leq \sqrt{n-m_n} ,
	\] so $\lceil\sqrt{i}\rceil \leq \sqrt{n} - \sqrt{\log{n}}$ implies $i \leq n-m_n$. It follows that for large $n$, whenever
	\begin{equation*}
	\left(\frac{a_nn^{1/4} + \pi - 2}{\pi}\right)^2 \leq i \leq \left(\sqrt{n} - \sqrt{\log{n}} - 1\right)^2,
	%\label{eqn:ifinalbds}
	\end{equation*}
	$i \leq n-m_n$ and conditions \eqref{eqn:iupper} and \eqref{eqn:ilower} hold, so via \eqref{eqn:MMprob},
	\begin{equation}
	\P\left(\abs{\lambda_i'(G_n/\sqrt{n})-\widetilde{\lambda}_i} > \frac{a_n}{n^{1/4}}\right) \ll n^{-\alpha-1}.
	\label{eqn:MMfinalProb}
	\end{equation}
	Define $J''_n \subset [n]$ to be the set of indices for which 
	\begin{equation}
	\left(\frac{a_nn^{1/4} + \pi - 2}{\pi}\right)^2 \leq \tau_n(i) \leq \left(\sqrt{n} - \sqrt{\log{n}} - 1\right)^2,
	\end{equation}
	does not hold. Then,
	\begin{equation}
	\abs{J''_n} \ll n^{o(1) + 1/2}
	\label{eqn:MMpart2a}
	\end{equation} and by the union bound applied to \eqref{eqn:MMfinalProb}, 
	\begin{equation}
	\max_{i \in [n] \setminus J''_n}\abs{\lambda'_{\tau_n(i)}(G_n/\sqrt{n})-\widetilde{\lambda}_{\tau_n(i)}} \ll n^{o(1)-1/4}
	\label{eqn:MMpart2b}
	\end{equation}
	with probability at least $1-C_\alpha n^{-\alpha}$ for a constant $C_\alpha > 0$ depending on $\alpha$. Since $\alpha > 0$ was arbitrary, the conclusion of Lemma \ref{lemma:MMpairs} follows by defining $J_n := J'_n \cup J''_n$ and combining \eqref{eqn:MMpart1a}, \eqref{eqn:MMpart1b}, \eqref{eqn:MMpart2a}, and \eqref{eqn:MMpart2b}.
\end{proof}

With Lemma \ref{lemma:MMpairs} in hand, we are ready to establish Theorems \ref{thm:finiteStats} and \ref{thm:infiniteStats}.  To that end, choose $\eta > 0$ small enough that $f$ is Lipschitz continuous in the disk $\set{z \in \C : \abs{z} \leq 1+\eta}$, and apply Theorem \ref{thm:iem:specRad} so that all eigenvalues of $X_n/\sqrt{n}$ are contained in the disk $\set{z \in \C : \abs{z} \leq 1+\eta/2}$. Let $E_n$ denote the event that the conclusions of Theorem \ref{thm:iem:specRad} and Lemma \ref{lemma:MMpairs} hold, and on this event, define the permutation $\tau_n:[n] \to [n]$ and $J_n \subset [n]$ as in Lemma \ref{lemma:MMpairs}. For completeness, on the complement of $E_n$, define $\tau_n$ to be the identity permutation and $J_n$ to be the empty set. For clarity, we also define the random variables $\mathcal A_n$ and $\mathcal B_n$ to be
\begin{align*}
%\mathcal{A}_n &:= \sum_{i \notin I_n}f(\lambda_i(X_n/\sqrt{n})) - \E\left[\sum_{i \notin I_n}f(\lambda_i(X_n/\sqrt{n}))\right],\\
%\mathcal{B}_n &:= \sum_{i \in I_n}f(\lambda_i(X_n/\sqrt{n})) - \E\left[\sum_{i \in I_n}f(\lambda_i(X_n/\sqrt{n}))\right],\\
%	\mathcal{A}_n &:= \sum_{i =1}^n\left[f(\lambda_i(X_n/\sqrt{n})) - \E f(\lambda_i(X_n/\sqrt{n}))\right],\\
\mathcal{A}_n &:= \sum_{i \in I_n}f(\widetilde{\lambda}_{\tau_n(i)}) - \E \left[ \sum_{i \in I_n}f(\widetilde{\lambda}_{\tau_n(i)}) \right],\\
%\mathcal{B}'_n &:= \sum_{i \in I'_n}f(\widetilde{\lambda}_{\tau_n(i)}) - \E\left[\sum_{i \in I'_n}f(\widetilde{\lambda}_{\tau_n(i)})\right],\\
\mathcal{B}_n &:= \sum_{i \in I_n} \left[f(\widetilde{\lambda}_{\tau_n(i)}) - f(\lambda_i(X_n/\sqrt{n}))\right] - \E\left[\sum_{i \in I_n}\left[f(\widetilde{\lambda}_{\tau_n(i)}) - f(\lambda_i(X_n/\sqrt{n}))\right]\right].
\end{align*}

Recall the definition of $S_n[f](X_n/\sqrt{n})$ from \eqref{def:Snf}.  We wish to determine the asymptotic behavior of %convergence in distribution of
\begin{equation} \label{eq:removefew}
\sum_{i \in [n]\setminus I_n} f(\lambda_i(X_n/\sqrt{n})) - \E \left[  \sum_{i \in [n]\setminus I_n} f(\lambda_i(X_n/\sqrt{n})) \right]  = S_n[f](X_n/\sqrt{n}) -\mathcal{A}_n + \mathcal{B}_n
\end{equation} 
and of 
\begin{equation} \label{eq:onlyfew}
\sum_{i \in I_n} f(\lambda_i(X_n/\sqrt{n})) - \E \left[ \sum_{i \in I_n} f(\lambda_i(X_n/\sqrt{n})) \right] = \mathcal{A}_n -\mathcal{B}_n,
\end{equation} 
which we will accomplish in several parts. First, we will show that $\mathcal{B}_n \to 0$ in probability as $n\to \infty$, and second, we will establish that $S_n[f](X_n/\sqrt{n})$ and $\mathcal{A}_n$ are independent. Afterwards, the proofs of Theorems \ref{thm:finiteStats} and \ref{thm:infiniteStats} will diverge. In particular, we will determine the limiting distribution to which $\mathcal{A}_n$ (resp., $\mathcal{A}_n/\sqrt{K_n}$) converges in law and apply Slutsky's theorem to establish the conclusion of Theorem \ref{thm:finiteStats} (resp., Theorem \ref{thm:infiniteStats}). %By Slutsky's theorem, and the Central , we will conclude the desired result.\todo{revise a bit?} 

The next few lemmas establish the joint limiting behavior of $S_n[f](X_n/\sqrt{n})$, $\mathcal{A}_n$, and $\mathcal{B}_n$.

%\todo{edit this lemma to have $K_n$ in place of $K$. In particular, why $1-o(1)$?}
\begin{lemma} 
$\mathcal{B}_n$ converges to zero in probability as $n \to \infty$.  
%On the event, $E_n$, if  we also have $I_n \subset [n] \setminus J_n$, then
%	\[
%	\abs{\mathcal{B}_n} \ll_{f} K_nn^{o(1)-1/4} + e^{{K_n^2}/{n}}\left(K_nn^{o(1)-1/4} + \frac{K_n\abs{J_n}}{n}\right),
%	\]
%	where the asymptotic inequality depends  on $f$ but is independent of $K_n$. In particular, $\mathcal{B}_n \to 0$ in distribution as $n\to \infty$.% since $E_n$ holds with overwhelming probability and $\P(I_n \subset [n] \setminus J_n\mid E_n) \to 1$ by Lemma \ref{lemma:NearBin} (indeed on the event $E_n$, $\abs{J_n} = o(n)$).
	\label{lemma:MMCto0}
\end{lemma}
%\begin{corollary}As $n\to \infty$, $\mathcal{B}_n \to 0$ in distribution.
%\end{corollary}
\begin{proof}
Let
\[ q_n := \sum_{i \in I_n} \left| f(\widetilde{\lambda}_{\tau_n(i)} ) - f( \lambda_i(X_n/\sqrt{n}) ) \right|. \] 
By Markov's inequality, it suffices to show that
$\E [q_n] \to 0$ 
as $n \to \infty$.  

We decompose 
\begin{equation} \label{eq:qndecomp} 
	\E[q_n] = \E[q_n \oindicator{E_n}] + \E[q_n \oindicator{E_n^c} ], 
\end{equation}
and bound each term separately.  To do so, we will now utilize that $f$ has a polynomial tail.  Indeed, this assumption implies that there exists a constant $C > 0$ and an integer $m \geq 1$ so that
\begin{equation} \label{eq:fbnd}
	|f(z)| \leq C (1 + |z|^m) 
\end{equation} 
for all $z \in \mathbb{C}$.  For the second term in \eqref{eq:qndecomp}, we apply the Cauchy--Schwarz inequality to obtain
\[ \E [q_n \oindicator{E_n^c} ] \leq \sqrt{\Prob(E_n^c)} \sqrt{ \E q_n^2 }. \]
Since $E_n$ holds with overwhelming probability, in order to bound this term it will suffice to show that $\E [q_n^2] \ll n^{O(1)}$, where the implicit constants will depend on the constant $C$ and the integer $m$.    To obtain this bound, we apply the Cauchy--Schwarz inequality again and \eqref{eq:fbnd} to get
\begin{align*} 
	\E q_n^2 &\leq K_n \E \left[ \sum_{i \in I_n} \left| f(\widetilde{\lambda}_{\tau_n(i)} ) - f( \lambda_i(X_n/\sqrt{n}) ) \right|^2 \right] \\
	&\ll K_n^2 + K_n^2 \E \| X_n \|_{op}^{2m}, 
\end{align*}
where $\| \cdot \|_{op}$ denotes the operator norm.  Here, we have exploited the fact that the spectral radius of any matrix is bounded above by the operator norm.  
By Lemma \ref{lemma:exnorm}, we obtain the bound of $\E q_n^2 \ll n^{O(1)}$, which shows that the second term on the right-hand side of \eqref{eq:qndecomp} converges to zero.  

We now bound $\E[q_n \oindicator{E_n}]$ by partitioning the sample space into events $E_n \cap \{ |I_n \setminus J_n| = j\}$ for $j=0, \ldots, K_n$.  Observe that
\begin{align*}
	\E[q_n \oindicator{E_n}] &= \sum_{j=0}^{K_n} \E\left[ q_n \oindicator{E_n} \indicator{|I_n \setminus J_n| = j} \right]  \\
		&\ll \sum_{j=0}^{K_n} ( n^{o(1) -1/4} j + (K_n - j) )\E \left[ \oindicator{E_n} \indicator{|I_n \setminus J_n| = j}  \right] \\
		&= n^{o(1) -1/4}  \E \left[ \sum_{j=0}^{K_n} j \oindicator{E_n} \indicator{|I_n \setminus J_n| = j} \right] \! + \E \left[ \sum_{j=0}^{K_n} (K_n - j) \oindicator{E_n} \indicator{|I_n \setminus J_n| = j} \right],
\end{align*}
where we used the Lipschitz continuity of $f$ and \eqref{eq:MMpairs}, which holds on the event $E_n$.  To bound these two terms, we now apply Lemma \ref{lemma:NearBin} and use the assumption that $K_n = O(n^{1/4 - \eps})$.  Indeed, by conditioning on $X_n$ (which also fixes $J_n$), we have
\begin{align*} 
	\E \left[ \sum_{j=0}^{K_n} j \oindicator{E_n} \indicator{|I_n \setminus J_n| = j} \right] &= \E \left[ \oindicator{E_n} \sum_{j=0}^{K_n}  j \E \left[ \indicator{|I_n \setminus J_n| = j} \mid X_n \right] \right] \\
	&\ll \E \left[ \oindicator{E_n} K_n \left( 1- \frac{ |J_n|}{n} \right) \right]. 
\end{align*}
Here, we used the towering property of the conditional expectations and the fact that the event $E_n$ is measurable with respect to the $\sigma$-algebra generated by the entries of the matrix $X_n$.  
By hypothesis $K_n = O(n^{1/4-\eps})$, and so
\[ n^{o(1) - 1/4}\E \left[ \sum_{j=0}^{K_n} j \oindicator{E_n} \indicator{|I_n \setminus J_n| = j} \right] = o(1). \] 
We similarly bound
\begin{align*} 
	 \E \left[ \sum_{j=0}^{K_n} (K_n - j) \oindicator{E_n} \indicator{|I_n \setminus J_n| = j} \right] &= \E \left[ \oindicator{E_n} \sum_{j=0}^{K_n} (K_n - j) \E \left[ \indicator{|I_n \setminus J_n| = j} \mid X_n \right] \right] \\
	 &\ll \E \left[ \oindicator{E_n} \left( K_n - K_n \left( 1 - \frac{|J_n|}{n} \right) \right) \right] \\
	 &= o (1), 
\end{align*}
where we used that $|J_n| \ll n^{3/4 + o(1)}$  on the event $E_n$.  
Combining the bounds above, we find $\E[q_n \oindicator{E_n}] = o(1)$, which completes the proof.  
\end{proof}

\begin{lemma}
	For each $n$, the (random) set $\tau_n(I_n)$ is independent from $X_n$ and has the same distribution as $I_n$. In particular, the random variables $S_n[f](X_n/\sqrt{n})$ and $\mathcal{A}_n$ are independent.
	\label{lemma:MMind}
\end{lemma}
\begin{proof}
	First, we show that $\tau_n(I_n)$ is independent from $X_n$. Indeed, suppose $S$ is a $K_n$-element subset of $[n]$ and $B$ is a Borel subset of $n\times n$ matrices with entries in $\C$. Recall that $I_n$ is independent from the entries of $X_n$ by hypothesis, and so $I_n$ is independent of the $\sigma$-algebra generated by $X_n$ and $\tau_n$ ($\tau_n$ is $X_n$-measurable by construction). Letting $S_n$ denote the set of permutations on $[n]$, it follows that
	\begin{align*}
	\P(\tau_n(I_n) = S,\ X_n \in B) & = \sum_{\sigma \in S_n}\P(\tau_n(I_n) = S,\ X_n \in B,\ \tau_n = \sigma)\\
	& = \sum_{\sigma \in S_n}\P(I_n = \sigma^{-1}(S),\ X_n \in B,\ \tau_n = \sigma)\\
	&= \sum_{\sigma \in S_n}\P(I_n = \sigma^{-1}(S))\cdot\P(X_n \in B,\ \tau_n = \sigma)\\
	&= \frac{1}{\binom{n}{K_n}} \sum_{\sigma \in S_n}\P(X_n \in B,\ \tau_n = \sigma)\\
	&=\frac{1}{\binom{n}{K_n}} \cdot \P(X_n \in B).
	\end{align*}
	We conclude that $\tau_n(I_n)$ and $X_n$ are independent and that $\tau_n(I_n)$ has the same distribution as $I_n$ since  
	\begin{align*}
	\P(\tau_n(I_n) = S) &=  \sum_{\sigma \in S_n}\P(I_n = \sigma^{-1}(S),\ \tau_n = \sigma)\\
	&= \sum_{\sigma \in S_n}\P(I_n = \sigma^{-1}(S))\cdot\P(\tau_n = \sigma)\\
	&= \frac{1}{\binom{n}{K_n}} \sum_{\sigma \in S_n}\P(\tau_n = \sigma)\\
	&= \frac{1}{\binom{n}{K_n}}.
	\end{align*}
	Furthermore, since $\{ \widetilde{\lambda}_i\}$ are deterministic, we observe that
	\[
	\mathcal{A}_n = \sum_{i \in \tau_n(I_n)}f(\widetilde{\lambda}_{i}) - \E\left[\sum_{i \in \tau_n(I_n)}f(\widetilde{\lambda}_{i})\right]
	\]
	is a function of only $\tau_n(I_n)$, while $S_n[f](X_n/\sqrt{n})$ is a function of only $X_n$, so $S_n[f](X_n/\sqrt{n})$ and $\mathcal{A}_n$ are independent.
\end{proof}

\begin{lemma} If $K_n = K$ is fixed, then
	\[
	\mathcal{A}_n \xrightarrow{n\to \infty} \sum_{i =1}^{K}\left( f(U_i) - \E \left[ f(U_i) \right] \right),
	\]
	in distribution, where $U_1, \ldots, U_K$ are iid complex random variables drawn from the uniform distribution $\mu_{\mathrm{disk}}$ on the unit disk. On the other hand, if $K_n \to \infty$ and $K_n = O(n^{1/4-\eps})$ for some fixed $\eps \in (0, 1/4)$, then,
	\[
	\frac{\mathcal{A}_n}{\sqrt{K_n}} \xrightarrow{n\to \infty} \mathcal{N} 
	\]
	in distribution, where $\mathcal{N}$ is a zero-mean complex normal distribution with covariances given in \eqref{eq:Ncov} and \eqref{eq:Ncov2}.  
	\label{lemma:MMdist}
\end{lemma}
\begin{proof}
	Since $K_n$ is much less than $n$, $\mathcal{A}_n$ acts like a sum of iid random variables, whose limiting behavior is much easier to understand. In order to state this observation more concretely, we will define a new index set $I'_n \subset [n]$, identical in distribution to $I_n$, but whose elements are coupled to a sequence of iid draws from $[n]$. To that end, let $Y_1, \ldots, Y_{K_n}$ be a sequence of iid uniformly random draws from $[n]$.  % chosen independently from the $\sigma$-algebra generated by $\{\widetilde{\lambda}_i\}$. 
	If $Y_1, \ldots, Y_{K_n}$ are distinct, we set $I_n' := \{Y_1, \ldots, Y_{K_n}\}$.  If they are not distinct, we take $I_n' := I_n$.  
%	Also, define the sequence $V_1, V_2, \ldots $, independently of $\{\widetilde{\lambda}_i\}$, via $V_1 = Y_1$ and, for $2 \leq i \leq K_n$,
%	\[
%	V_i := \begin{cases}
%	Y_i & \text{if $Y_1, \ldots, Y_i$ are distinct}\\
%	j \in [n] \setminus \set{V_1, \ldots, V_{i-1}} &\text{with prob.\ $\frac{1}{n- i+1}$ each otherwise}.		
%	\end{cases}
%	\] 
	From this construction, it is easy to check that $I_n'$ has the same distribution as $I_n$.  % and is independent of the $\sigma$-algebra generated by $\{ \widetilde{\lambda}_i\}$.
	Now, by Lemma \ref{lemma:MMind}, 
	%We will compare $\mathcal{A}_n$ to a similar sum that indexed over $K_n$ iid draws from $[n]$ rather than over $I_n$, so whose limiting distribution will be easier to calculate because the indices of the new sums will be drawn from $[n]$ with replacement . 
	$\tau(I_n)$ has the same distribution as $I_n$. It follows that $\mathcal{A}_n$ has the same distribution as
	\[
	\mathcal{A}'_n := \sum_{i \in I'_n}f(\widetilde{\lambda}_{i}) - \E \left[ \sum_{i \in I'_n}f(\widetilde{\lambda}_{i}) \right]. \\
	%\mathcal{B}'_n &:= \sum_{i \in I'_n}f(\widetilde{\lambda}_{\tau_n(i)}) - \E\left[\sum_{i \in I'_n}f(\widetilde{\lambda}_{\tau_n(i)})\right],\\
	\]
	We will now check that, with probability $1-O(1/\sqrt{n})$, $I_n' = \{Y_1, \ldots, Y_{K_n}\}$.  Indeed, we compute 
	\begin{align*}
	\P(\text{$Y_1, \ldots, Y_{K_n}$ are distinct}) &=\frac{n}{n}\cdot \frac{n-1}{n}\cdots \frac{n-K_n+1}{n}\\
	&\geq \left(\frac{n-K_n + 1}{n}\right)^{K_n}\\
	&\geq \exp\left[-K_n\cdot\frac{\frac{K_n-1}{n}}{\sqrt{1-\frac{K_n-1}{n}}}\right] \\
	&\geq 1 - O \left( \frac{1}{\sqrt{n} } \right),
	%& = \frac{1}{n^{K_n}}\cdot\frac{n!}{(n-K_n)!}\\
	%& \geq \frac{1}{n^{K_n}}\cdot \frac{\sqrt{2\pi}\ n^{n + 1/2}}{e^n}\cdot \frac{e^{n-K_n}}{e(n-K_n)^{n-K_n + 1/2}}\\
	%& = \frac{\sqrt{2\pi}}{e}\cdot \left[\sqrt{1-\frac{K_n}{n}}\cdot e^{K_n}\left(1-\frac{K_n}{n}\right)^{n-K_n}\right]^{-1}.
	\end{align*}
	where the second to last inequality follows from \eqref{eqn:KnDraws} above and the last inequality follows from the bound $e^{-x} \geq 1 - x$, valid for all $x \in \mathbb{R}$.  % Since $\P(F_n) \to 1$, 
	%		If we use a first-order approximation to the logarithm, 
	%		\[
	%		\left(1-\frac{K_n}{n}\right)^{n-K_n} = \exp\left[(n-K_n)\log\left(1-\frac{K_n}{n}\right)\right] \leq e^{-K_n(n-K_n)/n},
	%		\]
	%		so continuing from above,
	%		\[
	%		\P(F_n) \geq \frac{\sqrt{2\pi}}{e}\cdot \left[\sqrt{1-\frac{K_n}{n}}\cdot e^{K_n^2/n}\right]^{-1} \xrightarrow{n\to \infty} 1.
	%		\]
	%		which implies $\P(F_n) \gg e^{-K_n^2/n} \to 1$ as $n \to \infty$. 	
%	In our analysis of $\mathcal{A}'_n$, we will work on the event $$F_n := \set{I'_n = \set{Y_1, \ldots, Y_{K_n}}}$$ to take advantage of the fact that $Y_1, \ldots, Y_{K_n}$ are iid uniform draws from $[n]$. 
	%We now find the asymptotic distribution of $\mathcal{A}'_n$. By Lemma \ref{lem:distDraws} with $K_n := K$,
	%	\[
	%	\P\left(\tau_n(I_n) = \set{Y_1, \ldots, Y_K}\right) = \P(Y_1, \ldots, Y_K\ \text{are distinct}) \gg e^{-K^2/n}, % \xrightarrow{n\to \infty} 1,
	%	\]
	%so we will work
	%		On the event $$F_n := \set{\tau_n(I_n) = \set{Y_1, \ldots, Y_K}}$$ to take advantage of the fact that $Y_1, \ldots, Y_K$ are iid uniform draws from $[n]$ that are independent from $X_n$. (Note: by Lemma \ref{lemma:MMind}, $\tau_n(I_n)$ is independent of $X_n$.)
	By the assumptions on $K_n$, this implies that 
	\[ \E \left| \sum_{i \in I_n'} f( \widetilde{\lambda}_i ) - \sum_{i=1}^{K_n} f(\widetilde{\lambda}_{Y_i} ) \right| \longrightarrow 0 \]
	as $n \to \infty$, and so  
	with probability $1-o(1)$, 
	\begin{equation} \label{eq:suminq}
	\mathcal{A}'_n = \sum_{i = 1}^{K_n}\left(f(\widetilde{\lambda}_{Y_i}) - \E \left[ f(\widetilde{\lambda}_{Y_i}) \right] \right) + o(1).
	\end{equation}
	It thus suffices to study the convergence of the sum on the right-hand side of \eqref{eq:suminq}.  
	
	Recall that the values $\widetilde{\lambda}_1, \ldots, \widetilde{\lambda}_n$ are deterministic, and so the sum on the right-hand side of \eqref{eq:suminq} is a sum of iid random variables.  Moreover, it follows from Lemma 2.1 in \cite{MM}\footnote{The probability measure $\nu_n$ from Lemma 2.1 in \cite{MM} differs slightly from the distribution of $\widetilde{\lambda}_{Y_1}$.  The difference only involves $m_n/n$ proportion of the mass corresponding to the values $\widetilde{\lambda}_i$ for $n-m_n < i \leq n$.  In view of \eqref{MM:mbds}, $m_n/n \to 0$ as $n \to \infty$, so this discrepancy between the distributions is negligible in the limit.  } that $\widetilde{\lambda}_{Y_1}$ converges in distribution to a random variable with distribution $\mu_{\mathrm{disk}}$.   
%	Now, observe that $\widetilde{\lambda}_{Y_1}, \ldots, \widetilde{\lambda}_{Y_{K_n}}$ are iid draws from the distribution $\nu_n$ characterized by point masses of size $1/n$ at the points $\widetilde{\lambda}_1, \ldots, \widetilde{\lambda}_{n-m_n}$ and a mass of size $\frac{m_n}{n}$ distributed uniformly on the annulus $\set{z\in \C: \sqrt{1-\frac{m_n}{n}} \leq \abs{z} \leq 1}$ (Note: $\nu_n$ is defined similarly in \cite{MM}.) By Lemma 3 of \cite{MM}, $W_1(\nu_n, \mu_{\mathrm{disk}}) < 8/\sqrt{n}$, and since $\nu_n$ and $\mu_{\mathrm{disk}}$ have compact support, it follows by the Kantorovich--Rubinstein formulation of the Wasserstein distance that involves Lipschitz test functions (see e.g. \cite{V}, Remark 6.5 on page 95) that $\nu_n$ converges weakly to $\mu_{\mathrm{disk}}$. Consequently, using independence, we see that for any finite collection $Y_{i_1}, \ldots, Y_{i_j}$ of random variables from among $Y_1, Y_2, \ldots, Y_{K_n}$, the random vector $(\widetilde{\lambda}_{Y_{i_1}}, \ldots, \widetilde{\lambda}_{Y_{i_j}})$ converges in distribution to the vector $(U_1, \ldots, U_{j})$, where $U_1, \ldots, U_j$ are jointly independent random variables each having a uniform distribution on the unit disk in the complex plane. 
	Thus, in the case where $K_n = K$ is fixed, 
	\begin{equation*}
	\mathcal{A}'_n  \xrightarrow{n\to \infty} \sum_{i=1}^K\left(f(U_i)- \E \left[ f(U_i) \right] \right),
	\label{eqn:AnPrime1}
	\end{equation*}
	in distribution, where $U_1, \ldots, U_K$ are iid random variables with common distribution $\mu_{\mathrm{disk}}$.   Here, we have also exploited the fact that $f$ is bounded and continuous on the disk $\{z \in \mathbb{C} : |z| \leq 1\}$.  
	
	We now consider the case when $K_n \to \infty$, which will follow from the classical central limit theorem.  It will be slightly more convenient in this case to work with real-valued random variables.  Indeed, by the Cram\'{e}r--Wold device, it suffices to study the convergence of 
	\begin{equation} \label{eq:cwrv}
		\frac{1}{\sqrt{K_n} } \sum_{i = 1}^{K_n} \left( \alpha \left[ \Re f(\widetilde{\lambda}_{Y_i}) - \E \left[ \Re f(\widetilde{\lambda}_{Y_i})\right]\right] + \beta \left[ \Im f(\widetilde{\lambda}_{Y_i}) - \E \left[ \Im f(\widetilde{\lambda}_{Y_i}) \right] \right] \right) 
	\end{equation} 
	for arbitrary constants $\alpha, \beta \in \mathbb{R}$.  From the classical central limit theorem for triangular arrays (see, for instance, Theorem 3.4.5 in \cite{D}), it follows that the random variable in \eqref{eq:cwrv} converges in distribution to a mean zero normal random variable with variance 
	\[ \alpha^2 \var ( \Re f(U) ) + \beta^2 \var (\Im f(U)) + 2 \alpha \beta \cov( \Re f(U), \Im f(U) ), \]
	where $U$ has distribution $\mu_{\mathrm{disk}}$.  It follows from the Cram\'{e}r--Wold device that
	\begin{equation*}
	\frac{1}{\sqrt{K_n}}\mathcal{A}'_n \xrightarrow{n\to \infty} \mathcal{N} 
	\label{eqn:AnPrime2}
	\end{equation*}
	in distribution, 
	where $\mathcal{N}$ is the mean zero complex normal distribution with covariances defined in \eqref{eq:Ncov} and \eqref{eq:Ncov2}.  
\end{proof}

Recalling \eqref{eq:removefew} and \eqref{eq:onlyfew}, the proofs of Theorems \ref{thm:finiteStats} and \ref{thm:infiniteStats} now follow from Slutsky's theorem (see e.g.\ Lemma 1.10 in \cite{Petrov}) and Lemmas \ref{lemma:MMCto0}, \ref{lemma:MMind}, and \ref{lemma:MMdist}.

\appendix

\section{Local circular law} \label{sec:TV}

This section is devoted to the proof of Theorem \ref{thm:iem:TV}, which will follow from Theorem \ref{thm:loc:main} below.  Throughout this appendix, we let $\| f\|_p$ denote the $L^p$-norm of the function $f: \mathbb{C} \to \mathbb{C}$.  

\begin{theorem}[Local circular law] \label{thm:loc:main}
Fix $C > 0$, let $X_n$ be an $n \times n$ iid matrix whose atom distribution satisfies Assumption \ref{assump:moments}, and let $G_n$ be an $n \times n$ iid matrix drawn from the complex Ginibre ensemble. Then, with overwhelming probability,
\begin{equation}
\sum_{k=1}^n f(\lambda_k(X_n)) = \sum_{k=1}^n f(\lambda_k(G_n)) + O( \| \lap f \|_{1} n^{o(1)} ),
\label{eqn:locLaw}
\end{equation}
uniformly for all smooth functions $f: \mathbb{C} \to \mathbb{R}$ (possibly depending on $n$) that are supported on the disk $\{z \in \mathbb{C} : |z| \leq C \sqrt{n} \}$ and satisfy the bound
\begin{equation} \label{eq:loc:L3}
	\|\lap f \|_{3} \leq n^C \| \lap f \|_{1}. 
\end{equation}
 %where $\| \lap f \|_p$ is the $L^p$-norm of $\lap f$. 
\end{theorem}

% Original statement of theorem below:

%\begin{theorem}[Local circular law] \label{thm:loc:main}
%	Fix $C > 0$, and let $f: \mathbb{C} \to \mathbb{R}$ be a smooth function (possibly depending on $n$) supported on the disk $\{z \in \mathbb{C} : |z| \leq C \sqrt{n} \}$ which satisfies the bound
%	\begin{equation} \label{eq:loc:L3}
%	\|\lap f \|_{3} \leq n^C \| \lap f \|_{1}. 
%	\end{equation}
%	Let $X_n$ be an $n \times n$ iid matrix whose atom distribution satisfies Assumption \ref{assump:moments}, and let $G_n$ be an $n \times n$ iid matrix drawn from the complex Ginibre ensemble.  Then
%	\begin{equation}
%	\sum_{k=1}^n f(\lambda_k(X_n)) = \sum_{k=1}^n f(\lambda_k(G_n)) + O( \| \lap f \|_{1} n^{o(1)} )
%	\label{eqn:locLaw}
%	\end{equation}
%	with overwhelming probability, where $\| \lap f \|_p$ is the $L^p$-norm of $\lap f$.
%\end{theorem}

Theorem \ref{thm:loc:main} follows from the results of Alt, Erd\H{o}s, and Kr\"{u}ger in \cite{AEK2}. In particular, since the entries of $X_n$ and $G_n$ have the same variance, Theorem \ref{thm:loc:main} follows from the triangle inequality after applying \cite[Theorem 2.3]{AEK2} twice (once for $X_n/\sqrt{n}$ and once for $G_n/\sqrt{n}$). In each application, we set $a=0$ and $z_0=0$, and we consider the scaled functions $\tilde{f}(z):= f(\sqrt{n}z)$. To achieve the asymptotic factor $n^{o(1)}$ in \eqref{eqn:locLaw}, we note that the inequalities in \cite[Theorem 2.3]{AEK2} hold uniformly for all $n\in \mathbb{N}$, so $\eps$ may slowly converge to zero as a function of $n$.

Technically, \cite[Theorem 2.3]{AEK2} requires the additional assumption that the entries of $X_n$ have bounded density.  However, following Remark 2.5 in \cite{AEK2}, one can remove the bounded density assumption by utilizing a different bound on the least singular value.  For iid matrices, one can use the least singular value bound given by Tao and Vu  in \cite[Theorem 2.1]{TVcirc}.  Using this alternative bound requires a few small changes to the proof of \cite[Theorem 2.3]{AEK2}; these changes are explained in \cite[Remark 6.2]{AEK2} so we shall not repeat them here.    

While we have stated Theorem \ref{thm:loc:main} for iid matrices, the results in \cite{AEK2} are actually much more general and apply to a larger class of random matrices. Theorem \ref{thm:loc:main} should be compared to other local circular laws such as \cite[Theorem 20]{TV} and the main results in \cite{Y} (see also \cite{BYY, BYY2}).  Unfortunately, the random matrices under consideration here do not satisfy the assumptions of \cite[Theorem 20]{TV} since the real and imaginary parts of the entries of $X_n$ are not assumed to be independent.  In addition, \cite[Theorem 20]{TV} is only stated for open balls, while Theorem \ref{thm:loc:main} above will allow us to approximate squares (as well as other geometric shapes) by taking $f$ to be a smooth approximation to the indicator function.  Though very similar to Theorem \ref{thm:loc:main}, the results in \cite{Y} require the function $f$ to take a specific form (with a specific dependence on $n$), which differs slightly from what we use here.
%
%The main results in \cite{Y} are very similar to Theorem \ref{thm:loc:main}.  The results in \cite{Y} require the function $f$ to take a specific form (with a specific dependence on $n$), which differs slightly from what we use here. 

\subsection{Proof of Theorem \ref{thm:iem:TV}}
It remains to prove Theorem \ref{thm:iem:TV}.  The proof relies on Theorem \ref{thm:loc:main} above and some results from \cite{TV}.  Before presenting the proof, we introduce the following notation.  For $z_0 \in \mathbb{C}$ and $r \geq 0$, let 
\[ B(z_0, r) := \{ z \in \mathbb{C} : |z - z_0| < r \} \] 
be the open ball of radius $r$ centered at $z_0$.

We now turn to the proof of Theorem \ref{thm:iem:TV}.   
Let $C' > 1$ be a large fixed constant so that the square $R$ is contained in $B(0, (C'-1) \sqrt{n})$.  From \cite[Theorem 20]{TV}, it follows that
\begin{equation} \label{eq:loc:errcnttilde}
	\widehat{N}(B(z_0, 1)) \ll n^{o(1)} 
\end{equation}
with overwhelming probability, uniformly for all $z_0 \in B(0, C' \sqrt{n})$.  Unfortunately, the matrix $X_n$ does not satisfy the assumptions of \cite[Theorem 20]{TV}, so we will need to use Theorem \ref{thm:loc:main} to obtain a version of \eqref{eq:loc:errcnttilde} for the eigenvalues of $X_n$.  Indeed, letting $f:\mathbb{C} \to [0,1]$ be a smooth approximation to the indicator function on $B(z_0, 1)$, we find from Theorem \ref{thm:loc:main} and \eqref{eq:loc:errcnttilde} that
\begin{equation} \label{eq:loc:errorcnt}
	N(B(z_0, 1)) \ll n^{o(1)} 
\end{equation}
with overwhelming probability, uniformly for all $z_0 \in B(0, C' \sqrt{n})$.  We will return to these bounds in a moment.  

For each $1 \leq \ell \leq L$, let $f_{\ell}:\mathbb{C} \to [0,1]$ be a smooth approximation to the indicator function on $R_\ell$ so that $f_\ell(z) = 1$ for $z \in R_{\ell}$ and $f_\ell(z) = 0$ for $z$ at distance $1$ or more from $R_{\ell}$.  The functions $f_\ell$ can be chosen to satisfy \eqref{eq:loc:L3} (for a suitably large choice of constant $C > 0$) and so that
\[ \max_{\ell} \sup_{z \in \mathbb{C} } |\lap f_{\ell}(z)| = O(1). \] 
A geometric argument that uses this bound for regions where $\abs{\lap f_\ell(z)} > 0$ shows that $\max_{\ell} \| \lap f_{\ell} \|_1 = O(n^{1/4})$.  Applying Theorem \ref{thm:loc:main} and the union bound, we conclude that
\begin{equation} \label{eq:loc:sumf}
	\max_{\ell} \left| \sum_{k=1}^n f_{\ell} (\lambda_k(X_n)) - \sum_{k=1}^n f_{\ell} (\lambda_k(G_n)) \right| \ll n^{1/4 + o(1)} 
\end{equation}
with overwhelming probability.  

We now go from the sums above to the counting functions $N(R_{\ell})$ and $\widehat{N}(R_{\ell})$.  For $1 \leq \ell \leq L$, let 
\[ T_{\ell} := \{z \in \mathbb{C} \setminus R_\ell : f_{\ell}(z) > 0 \}. \]
By construction $T_{\ell}$ is disjoint from $R_{\ell}$.  Moreover, all the points in $T_{\ell}$ are distance at most $1$ from $R_{\ell}$.  
By covering $T_{\ell}$ with unit balls, we can apply \eqref{eq:loc:errcnttilde}, \eqref{eq:loc:errorcnt}, and the union bound to obtain 
\begin{equation} \label{eq:loc:max}
	\max_{\ell} \left( N(T_{\ell}) + \widehat{N}(T_{\ell}) \right)  \ll n^{1/4+o(1)} 
\end{equation}
with overwhelming probability.  

Combining \eqref{eq:loc:sumf} and \eqref{eq:loc:max}, we conclude that
\begin{align*}
	\max_{\ell} | N(R_{\ell}) - \widehat{N}(R_{\ell}) | \!\!\; &\leq \max_{\ell}\! \left[ \left| \sum_{k=1}^n \!\!\; f_{\ell}(\lambda_k(X_n)) - \!\sum_{k=1}^n \!\!\; f_{\ell}(\lambda_k(G_n)) \right| \!\!\; + N(T_{\ell}) + \widehat{N}(T_{\ell}) \right] \\
	&\ll n^{1/4 + o(1)} 
\end{align*}
with overwhelming probability.  This completes the proof of Theorem \ref{thm:iem:TV}.

\bibliography{partStatsIID}{}
\bibliographystyle{abbrv}

\end{document}